\documentclass[12pt]{article}

\usepackage{amssymb}
\usepackage{eucal}
\usepackage{amsmath}
\usepackage{amsthm}
\usepackage{amscd}
\usepackage{graphics}
\usepackage{graphicx}
\usepackage{amsfonts}
\usepackage{latexsym}
\usepackage[all]{xy}

\newcommand{\R}{{\mathbb  R}}

\numberwithin{equation}{section}
\newtheorem{thm}{\bf Theorem}[section]

\newtheorem{defn}{\bf Definition}[section]

\theoremstyle{remark}
\newtheorem{rem}{\bf Remark}[section]

\begin{document}

\title{Double bracket structures on Poisson manifolds}
\author{Petre Birtea\\
{\small Departamentul de Matematic\u a, Universitatea de Vest din Timi\c soara}\\
{\small Bd. V. P\^arvan, Nr. 4, 300223 Timi\c soara, Rom\^ania}\\
{\small E-mail: birtea@math.uvt.ro}}
\date{}

\maketitle


\begin{abstract}
\noindent On a Poisson manifold endowed with a Riemannian metric we will construct a vector field that generalizes the double bracket vector field defined on semi-simple Lie algebras. On a regular symplectic leaf we will construct a generalization of the normal metric such that the above vector field restricted to the symplectic leaf is a gradient vector field with respect to this metric.
\end{abstract}

\section{Introduction}

We will recall the classical case of double bracket vector field on a semi-simple Lie algebra.
Let $(\frak g,[,])$ be a semi-simple Lie algebra with $\bf k:{\frak g}\times {\frak g}\to \mathbb{R}$ the Killing form, i.e. a non-degenerate, symmetric, Ad-invariant bilinear form. The following vector field, called the double bracket vector field, has been introduced by Brockett \cite{brockett-1}, \cite{brockett-2}, see also \cite{bloch}, in the context of dynamical numerical algorithms and linear programming
\begin{equation}\label{1}
\dot L=[L,[L,N]],
\end{equation}
where $N\in {\frak g}$ is a regular element in ${\frak g}$ and $L\in {\frak g}$. It turns out that the double bracket vector field is tangent to the adjoint orbits of the Lie algebra ${\frak g}$, orbits that are symplectic leaves for the canonical Lie-Poisson bracket on ${\frak g}$ induced by the K--K--S Poisson bracket on ${\frak g}^*$ and the identification between ${\frak g}$ and ${\frak g}^*$ being given by the Killing form.

More precisely, by identifying the Lie algebra ${\frak g}$ with its dual ${\frak g}^*$ using the Killing form, the K--K--S Poisson bracket on ${\frak g}^*$ transforms into the Poisson bracket on ${\frak g}$,
\begin{equation}\label{2}
\{F,G\}_{{\frak g}}(L)={\bf k}(L,[\nabla F(L),\nabla G(L)]),
\end{equation}
where $F,G:{\frak g}\to \mathbb{R}$ are smooth functions. Also, the Hamiltonian vector field associated to a smooth function $H:{\frak g}\to \mathbb{R}$ is given by
$$\dot L=-[L,\nabla H(L)],$$
see \cite{ratiu-1}, \cite{ratiu-2} and \cite{vanhaecke} for a pedagogical exposition of the above construction.

In the case of a compact Lie algebra ${\frak g}$, it has been proved that the double bracket vector field \eqref{1} when restricted to a regular adjoint orbit $\Sigma\subset {\frak g}$ is a gradient vector field with respect to the normal metric. We recall briefly this construction, for details see \cite{bloch-brockett}, \cite{bloch-flaschka}. For every $L\in \Sigma$ consider the orthogonal decomposition with respect to the minus Killing form $\bf k$, ${\frak g}={\frak g}_L\oplus {\frak g}^L$, where ${\frak g}_L:=\{X\in {\frak g}|~[L,X]=0\}$ and ${\frak g}^L=Im(ad_L)$. The linear space ${\frak g}_L$ can be identified with the tangent space $T_L\Sigma$ and ${\frak g}^L$ with the normal space. One can endow the adjoint orbit $\Sigma$ with the normal metric \cite{besse}, or standard metric \cite{atiyah},
$${\bf n}^{\Sigma}([L,X],[L,Y])=-{\bf k}(X^L,Y^L),$$
where $X^L,Y^L$ are the normal components according to the above orthogonal decomposition of $X$, respectively $Y$.
\begin{thm} [\cite{bloch-brockett}, \cite{bloch-flaschka}] \label{bloch-ratiu}
Let $H:{\frak g}\to \mathbb{R}, H(L)={\bf k}(L,N)$. Then
$$\nabla_{{\bf n}^{\Sigma}}H_{|{\Sigma}}(L)=[L,[L,N]] .$$
\end{thm}

The purpose of this paper is to obtain an analog of the above theorem in the setting of a general Poisson manifold $(M,\{\cdot,\cdot\})$. For this, first we need to construct a generalized double bracket vector field tangent to the symplectic leaves of an arbitrary Poisson manifold that corresponds to the double bracket vector field \eqref{1} for the case when Poisson manifold is ${\frak g}$ endowed with the Poisson bracket \eqref{2}. Following an idea presented in \cite{birtea-comanescu} we will construct a co-metric tensor that couples the Poisson structure and the Riemannian structure. Next step is to construct a generalized normal Riemannian metric on regular symplectic leaves such that the restriction of the generalized double bracket vector field to a regular leaf is the gradient of a smooth function with respect to this generalized normal Riemannian metric.

For the case of a compact Lie algebra ${\frak g}$ we will rediscover the geometry of the double bracket vector field, where ${\frak g}$ is endowed with the Poisson structure \eqref{2}. We will exemplify our construction on the case of the non-compact semi-simple Lie algebra $sl(2,\mathbb{R})$.

\section{Generalized double bracket vector field on Poisson manifolds}

Let $(M,{\bf g})$ be a (pseudo-)Riemannian manifold that is also endowed with a Poisson bracket
$$\{\cdot,\cdot\}:{\cal C}^{\infty}(M)\times {\cal C}^{\infty}(M)\to {\cal C}^{\infty}(M).$$

\begin{defn}
We call the {\bf co-metric double bracket tensor} the following symmetric contravariant 2-tensor ${\bf D}:\Omega^1(M)\times \Omega^1(M)\to {\cal C}^{\infty}(M)$,
$${\bf D}(\alpha,\beta):={\bf g}(\#_{\bf {\Pi}}\alpha,\#_{\bf {\Pi}}\beta ),$$
where ${\bf \Pi}$ is the skew-symmetric contravariant 2-tensor associated with the bracket $\{\cdot ,\cdot \}$.
\end{defn}
When $\alpha =dF$ and $\beta = dG$, where $F,G\in {\cal C}^{\infty}(M)$ we have
$${\bf D}(dF,dG)={\bf g}(X_F,X_G).$$
In the finite dimensional case the symmetric matrix associated to the contravariant tensor ${\bf D}$ is given by
$$[{\bf D}]=[{\bf \Pi}]^T\cdot[{\bf g}]\cdot[{\bf \Pi}].$$
The 2-tensor ${\bf D}$ is degenerate and its kernel is equal with the kernel of the Poisson bivector ${\bf \Pi}$ in the case of a Riemannian metric on $M$.

We introduce the generalization of the double bracket vector field. In a different context, a similar construction that used a co-metric tensor has also been given in \cite{birtea-comanescu}.
\begin{defn}
For $G\in {\cal C}^{\infty}(M)$, the vector field
$${\bf v}_G:=-{\bf i}_{dG}{\bf D}$$
is called the {\bf generalized double bracket vector field}.
\end{defn}

The vector field ${\bf v}_G$ is a natural generalization of the double bracket vector field defined on a semi-simple Lie algebra.
\begin{thm}\label{egalitate}
Let $({\frak g},[\cdot,\cdot])$ be a semi-simple Lie algebra endowed with the Poisson bracket defined by \eqref{2}. Then,
$${\bf v}_G(\xi)=[\xi,[\xi,\nabla G(\xi)]].$$
\end{thm}

\begin{proof}
Let $\{{\bf e}_i\}_{i=\overline{1,n}}$ be a base for ${\frak g}$ and $\xi^i:{\frak g}\to \mathbb{R},i=\overline{1,n}$ the corresponding coordinate functions. Then,
$${\bf v}_G(\xi)=-\left({\bf i}_{dG}{\bf D}\right)(\xi)=-{\bf D}^{ij}(\xi)\frac{\partial G}{\partial \xi^j}{\bf e}_i,$$
where
\begin{align*}
-{\bf D}^{ij}(\xi)&=-{\bf k}\left(X_{\xi^i}(\xi),X_{\xi^j}(\xi)\right)=-{\bf k}\left([\xi,\nabla \xi^i(\xi)],[\xi,\nabla \xi^j(\xi)]\right)\\
&=-{\bf k}\left([\xi^p{\bf e}_p,k^{\alpha i}{\bf e}_{\alpha}],[\xi^s{\bf e}_s,k^{\beta j}{\bf e}_{\beta}]\right)\\
&=-{\bf k}\left(C_{p\alpha}^{\gamma}k^{\alpha i}\xi^p{\bf e}_{\gamma},C_{s\beta}^{\tau}k^{\beta j}\xi^s{\bf e}_{\tau}\right)\\
&=-k_{\gamma\tau}C_{p\alpha}^{\gamma}C_{s\beta}^{\tau}k^{\alpha i}k^{\beta j}\xi^p\xi^s\\
&\stackrel{(**)}{=}k_{\alpha\gamma}C_{p\tau}^{\gamma}C_{s\beta}^{\tau}k^{\alpha i}k^{\beta j}\xi^p\xi^s\\
&=\delta _{\gamma i}C_{p\tau}^{\gamma}C_{s\beta}^{\tau}k^{\beta j}\xi^p\xi^s\\
&=C_{p\tau}^iC_{s\beta}^{\tau}k^{\beta j}\xi^p\xi^s.
\end{align*}
Consequently,
$${\bf v}_G(\xi)=C_{p\tau}^iC_{s\beta}^{\tau}k^{\beta j}\xi^p\xi^s\frac{\partial G}{\partial \xi^j}{\bf e}_i.$$
We have the following computation,
\begin{align*}
[\xi, [\xi,\nabla G(\xi)]]&=\left[\xi^p{\bf e}_p,\left[\xi^s{\bf e}_s,k^{\beta j}\frac{\partial G}{\partial \xi^j}{\bf e}_{\beta}\right]\right]\\
&=\left[\xi^p{\bf e}_p,C_{s\beta}^{\tau}k^{\beta j}\xi^s\frac{\partial G}{\partial \xi^j}{\bf e}_{\tau}\right]\\
&=C_{p\tau}^i C_{s\beta}^{\tau}k^{\beta j}\xi^p\xi^s\frac{\partial G}{\partial \xi^j}{\bf e}_i.
\end{align*}

For the equality (**) we used the bi-invariance of the Killing metric, ${\bf k}([X,Y],Z)+{\bf k}(Y,[X,Z])=0$. Applying this equality to the three vectors ${\bf e}_p,{\bf e}_{\alpha},{\bf e}_{\tau}$ we have
$${\bf k}\left([{\bf e}_p,{\bf e}_{\alpha}],{\bf e}_{\tau}\right)+{\bf k}\left({\bf e}_{\alpha},[{\bf e}_p,{\bf e}_{\tau}]\right)=0$$
$$\Leftrightarrow {\bf k}\left(C_{p\alpha}^{\gamma}{\bf e}_{\gamma},{\bf e}_{\tau}\right)+{\bf k}\left({\bf e}_{\alpha},C_{p\tau}^{\gamma}{\bf e}_{\gamma}\right)=0
\Leftrightarrow -k_{\gamma\tau}C_{p\alpha}^{\gamma}=k_{\alpha\gamma}C_{p\tau}^{\gamma}.$$
\end{proof}

On a symplectic leaf $\Sigma\subset M$ we will construct a (pseudo-)Riemannian metric that will generalize the normal metric from the case of a compact semi-simple Lie algebra.

\begin{defn}\label{double-bracket-metric}
Let $x\in\Sigma$ and $X_{tan},Y_{tan}\in {\cal X}(\Sigma)$, then {\bf the double bracket metric} $\boldsymbol{\tau}_{db}^{\Sigma}:{\cal X}(\Sigma)\times {\cal X}(\Sigma)\to {\cal C}^{\infty}(\Sigma)$ is defined as
$$\boldsymbol{\tau}_{db}^{\Sigma}(x)\left(X_{tan}(x),Y_{tan}(x)\right):=({\bf g}^{\Sigma}_{ind})^{-1}(x)\left({\bf i}_{X_{tan}}\omega(x),{\bf i}_{Y_{tan}}\omega(x)\right),$$
where $\omega\in \Omega^2(\Sigma)$ is the induced symplectic 2-form on the symplectic leaf $\Sigma$ and $({\bf g}^{\Sigma}_{ind})^{-1}$ is the co-metric tensor associated to the (pseudo-)Riemannian metric ${\bf g}^{\Sigma}_{ind}$ induced on $\Sigma$ by the ambient (pseudo-)Riemannian metric $\bf g$.
\end{defn}

For $x\in \Sigma$, the matrix associated with the co-metric tensor $({\bf g}^{\Sigma}_{ind})^{-1}$ is given by
$$[{\bf g}^{\Sigma}_{ind}(x)]^{-1}=[[{\bf g}(x)]_{_{|T_x\Sigma \times T_x\Sigma}}]^{-1},$$
and consequently, the matrix associated with the double bracket metric $\boldsymbol{\tau}_{db}^{\Sigma}$ has the expression
$$[\boldsymbol{\tau}_{db}^{\Sigma}(x)]=[\omega(x)]^{T}\cdot[[{\bf g}(x)]_{_{|T_x\Sigma \times T_x\Sigma}}]^{-1}\cdot[\omega(x)].$$
\medskip

Note that in the case of a pseudo-Riemannian metric ${\bf g}$ the co-metric tensor $({\bf g}^{\Sigma}_{ind})^{-1}$ does not always exist.

It has been proved by Weinstein \cite{weinstein} that locally any Poisson manifold $(M,\{\cdot,\cdot\})$ is diffeomorphic to a product $\Sigma\times N$, where $\Sigma$ is a symplectic leaf of $M$
and $N\subset M$ is a transverse submanifold to $\Sigma$. Moreover, around any point $x\in M$ there exist a local system of coordinates $({\bf q},{\bf p},{\bf z})$, called Darboux-Weinstein coordinates, such that the symplectic leaf $\Sigma$ that contains the point $x$ is locally described by ${\bf z}={\bf 0}$ and ${\bf z}$ are local coordinates on the transverse submanifold $N$. In this set of coordinates the matrix associated with the Poisson tensor $\Pi$ has the local expression
$$[{\bf \Pi}]=\left(\begin{array}{c|c}
\left.\begin{array}{cc}0&\mathbb{I}_n\\-\mathbb{I}_n&0\end{array}\right. &0\\
\hline\\
0&{\Pi}^{'}({\bf z})
\end{array}\right),$$
where the matrix $[{\Pi}^{'}({\bf z})]$ corresponds to a Poisson bivector ${\Pi}^{'}$ that endows the submanifold $N$ with a Poisson structure called the transverse Poisson structure. Also, $\Pi^{'}({\bf 0})={\bf 0}$ and in the case when $\Sigma$ is a regular symplectic leaf we have $\Pi^{'}({\bf z})={\bf 0}$, for any $\bf z$ in the domain of Darboux-Weinstein coordinates. The transverse Poisson structure has been extensively studied in \cite{damianou}, \cite{cushman}, \cite{cruz}.

\medskip

In analogy with the compact semi-simple Lie algebra case we have the following result.
\begin{thm}\label{gradient th}
The generalized double bracket vector field is tangent to regular symplectic leaves and the restriction to a regular leaf $\Sigma$ is the gradient vector field of the restricted function $G_{|\Sigma}$ with respect to the double bracket metric. More precisely,
$$({\bf v}_G)_{|\Sigma}=-\nabla _{\boldsymbol{\tau}_{db}^{\Sigma}}G_{|\Sigma}.$$
\end{thm}
\begin{proof}
Locally the symplectic leaf $\Sigma$ is given by $\bf z=0$, where $(\bf q,\bf p,\bf z)$ is a set of Darboux--Weinstein coordinates. For an arbitrary point $(\bf q,\bf p,\bf 0)\in\Sigma$ in the domain of the Darboux--Weinstein coordinates, we have the following matrices for the corresponding tensors ${\bf g},{\bf \Pi}$, ${\bf g}^{\Sigma}_{ind}$, and $\omega$:
$$[{\bf g}]=\left(\begin{array}{ccc}g_{11}({\bf q},{\bf p},{\bf z})&g_{12}({\bf q},{\bf p},{\bf z})&g_{13}({\bf q},{\bf p},{\bf z})\\
g_{12}^T({\bf q},{\bf p},{\bf z})&g_{22}({\bf q},{\bf p},{\bf z})&g_{23}({\bf q},{\bf p},{\bf z})\\
g_{13}^T({\bf q},{\bf p},{\bf z})&g_{23}^T({\bf q},{\bf p},{\bf z})&g_{33}({\bf q},{\bf p},{\bf z})
\end{array}\right);~
[{\bf \Pi}]=\left(\begin{array}{c|c}
\left.\begin{array}{cc}0&\mathbb{I}_n\\-\mathbb{I}_n&0\end{array}\right. &0\\
\hline\\
0&0
\end{array}\right);$$
$$[{\bf g}^{\Sigma}_{ind}({\bf q},{\bf p})]=
\left(\begin{array}{cc}
A_{\bf q\bf q}(\bf q,\bf p)&A_{\bf q\bf p}(\bf q,\bf p)\\
A_{\bf q\bf p}^T(\bf q,\bf p)&A_{\bf p\bf p}(\bf q,\bf p)\\
\end{array}\right);~[\omega]=\left(\begin{array}{cc}
0&-\mathbb{I}_n\\
\mathbb{I}_n&0
\end{array}
\right),$$
with $A_{{\bf q q}}({\bf q, p})=A_{\bf q\bf q}^T({\bf q, p})=g_{11}({\bf q},{\bf p},{\bf 0});~A_{{\bf p p}}({\bf q, p})=A_{{\bf p p}}^T({\bf q, p})=g_{22}({\bf q},{\bf p},{\bf 0});~A_{{\bf q}{\bf p}}({\bf q},{\bf p})=g_{12}({\bf q},{\bf p},{\bf 0})$.
The vector field ${\bf v}_G$ has the local expression,
\begin{align*}
{\bf v}_G(\bf q,\bf p,\bf 0)&=\left(A_{{\bf q p}}^T({\bf q, p})\frac{\partial G}{\partial \bf p}-A_{{\bf p p}}({\bf q, p})\frac{\partial G}{\partial \bf q}\right)\frac{\partial}{\partial \bf q}\\
&+\left(-A_{{\bf q q}}({\bf q, p})\frac{\partial G}{\partial \bf p}+A_{{\bf q p}}(\bf q,\bf p)\frac{\partial G}{\partial \bf q}\right)\frac{\partial}{\partial \bf p},
\end{align*}
which is a tangent vector to the regular symplectic leaf $\Sigma$.

The vector field $\nabla_{\boldsymbol{\tau}_{db}^{\Sigma}}G_{|\Sigma}\in {\cal X}(\Sigma)$ has the following local expression,
\begin{align*}
-\nabla_{\boldsymbol{\tau}_{db}^{\Sigma}}G_{|\Sigma}(\bf q,\bf p)&=-\left([\omega]^T\cdot [{\bf g}^{\Sigma}_{ind}(\bf q,\bf p)]^{-1}\cdot [\omega]\right)^{-1}\left(\begin{array}{c}
\frac{\partial G}{\partial \bf q}\\
\\
\frac{\partial G}{\partial \bf p}
\end{array}\right)\\
&=-\left([\omega]^T\cdot [{\bf g}^{\Sigma}_{ind}(\bf q,\bf p)]\cdot[\omega]\right)\left(\begin{array}{c}
\frac{\partial G}{\partial \bf q}\\
\\
\frac{\partial G}{\partial \bf p}
\end{array}\right)\\
&=\left(A_{{\bf q p}}^T({\bf q, p})\frac{\partial G}{\partial \bf p}-A_{{\bf p p}}({\bf q, p})\frac{\partial G}{\partial \bf q}\right)\frac{\partial}{\partial \bf q} \\
&+\left(-A_{{\bf q q}}({\bf q, p})\frac{\partial G}{\partial \bf p}+A_{{\bf q p}}({\bf q, p})\frac{\partial G}{\partial \bf q}\right)\frac{\partial}{\partial \bf p}.
\end{align*}
\end{proof}
\begin{rem}
In analogy with the result in \cite{birtea-comanescu}, we have that for any $x\in\Sigma$,
$$\left({\bf D}(x)_{|T_{x}^*\Sigma\times T_{x}^*\Sigma}\right)^{-1}=\boldsymbol{\tau}_{db}^{\Sigma}(x).$$
More precisely, for an arbitrary $x=({\bf p},{\bf q},{\bf 0})\in \Sigma$ in the domain of Darboux--Weinstein local coordinates,
$$\left[{\bf D}(x)_{|T_x^*\Sigma\times T_x^*\Sigma}\right]=\left(\begin{array}{cc}
A_{\bf p\bf p}(\bf q,\bf p)&-A_{\bf q\bf p}^T(\bf q,\bf p)\\
-A_{\bf q\bf p}(\bf q,\bf p)&A_{\bf q\bf q}(\bf q,\bf p)
\end{array}\right)=[\boldsymbol{\tau}_{db}^{\Sigma}(x)]^{-1}.$$
\end{rem}
\medskip
Next, we will show that for the case of a compact semi-simple Lie algebra the double bracket metric and the normal metric coincide up to a sign.
\begin{thm}
Let ${\frak g}$ be a semi-simple compact Lie algebra and $\Sigma\subset {\frak g}$ a regular adjoint orbit. Then
$$\boldsymbol{\tau}_{db}^{\Sigma}=-{\bf n}^{\Sigma}.$$
\end{thm}
\begin{proof}
Let $({\bf p},{\bf q},{\bf z})$ be a system of Darboux--Weinstein local coordinates adapted to the regular adjoint orbit $\Sigma$ and we make the notations ${\bf u}=({\bf p},{\bf q})$. If $x_0=({\bf u}_0,{\bf 0})$ is an arbitrary point of $\Sigma$ that belongs to the domain of the adapted local coordinates system, then by Theorem \ref{gradient th} we have
$$[\boldsymbol{\tau}_{db}^{\Sigma}({\bf u}_0)]^{-1}\cdot du^i({\bf u}_0)=\nabla_{\boldsymbol{\tau}_{db}^{\Sigma}}u^i({\bf u}_0)=-{\bf v}_{u^i}({\bf u}_0,{\bf 0}),~~~\mbox{all}~~~i=\overline{1,\dim \Sigma}.$$
By Theorem \ref{bloch-ratiu} and Theorem \ref{egalitate} we also have
\begin{align*}
[{\bf n}^{\Sigma}({\bf u}_0)]^{-1}\cdot du^i({\bf u}_0)&=\nabla_{{\bf n}^{\Sigma}}u^i({\bf u}_0)=[x_0,[x_0,\nabla u^i(x_0)]]\\
&={\bf v}_{u^i}({\bf u}_0,{\bf 0}),~~~\mbox{all}~~~i=\overline{1,\dim \Sigma}.
\end{align*}
As $\{du^i({\bf u}_0)|~i=\overline{1,\dim\Sigma}\}$ is a base for $T_{{\bf u}_0}^*\Sigma$ we obtain the equality
$$[\boldsymbol{\tau}_{db}^{\Sigma}({\bf u}_0)]=-[{\bf n}^{\Sigma}({\bf u}_0)].$$
\end{proof}

On a regular symplectic leaf $\Sigma$, a sufficient condition for the matrix $[[{\bf g}(x)]_{_{|T_x\Sigma \times T_x\Sigma}}]^{-1}$ to be equal with the matrix $[{\bf g}^{-1}(x)]_{_{|T_x\Sigma \times T_x\Sigma}}$ is the compatibility condition that the Poisson bivector ${\bf \Pi}$ is ${\bf g}$-parallel, i.e.
$$\nabla {\bf \Pi}=0,\leqno{({\bf C})}$$
where $\nabla$ is the covariant derivative on the Riemannian manifold $(M,{\bf g})$, see \cite{izu}.

\section{Lie algebra $sl(2,\mathbb{R})$ and hyperbolic geometry}
In this section we show the connection between the structure of the semi-simple non-compact Lie algebra $sl(2,\mathbb{R})$ and the Poincar\'{e} open disc model for the hyperbolic geometry. More precisely, we prove that double bracket metric $\boldsymbol{\tau}_{db}^{\mathbb{H}^2}$ on the connected component of the adjoint orbit (symplectic leaf) given by the upper-sheet $\mathbb{H}^2$ of the two-sheeted hyperboloid is the hyperbolic metric in the Poincar\'{e} open disc model.

A base for Lie algebra $sl(2,\mathbb{R})$ is given by
$${\bf e}_1=\left(\begin{array}{cc}
0&1\\
0&0
\end{array}\right);~~~{\bf e}_2=\left(\begin{array}{cc}
0&0\\
1&0
\end{array}\right);~~~{\bf e}_3=\left(\begin{array}{cc}
1&0\\
0&-1
\end{array}\right).$$
For the Lie algebra structure we have
$$[{\bf e}_1,{\bf e}_2]={\bf e}_3;~~~[{\bf e}_1,{\bf e}_3]=-2{\bf e}_1; ~~~[{\bf e}_2,{\bf e}_3]=2{\bf e}_2,$$
and the associated adjoint operators 
$$\mbox{ad}~{\bf e}_1=\left(\begin{array}{ccc}
0&0&-2\\
0&0&0\\
0&1&0
\end{array}\right);~~~\mbox{ad}~{\bf e}_2=\left(\begin{array}{ccc}
0&0&0\\
0&0&2\\
-1&0&0
\end{array}\right);~~~\mbox{ad}~{\bf e}_3=\left(\begin{array}{ccc}
2&0&0\\
0&-2&0\\
0&0&0
\end{array}\right).$$
The matrix corresponding to the Killing metric in this base is given by
$$[{\bf k}]=\left(\begin{array}{ccc}
0&4&0\\
4&0&0\\
0&0&8
\end{array}\right).$$
We make the following change of base
$${\bf e}_x =\frac{1}{2\sqrt{2}}({\bf e}_1 +{\bf e}_2);~~~
{\bf e}_y=\frac{1}{2\sqrt{2}}{\bf e}_3;~~~
{\bf e}_ z=\frac{1}{2\sqrt{2}}({\bf e}_1 -{\bf e}_2 ).$$
For the Lie bracket we have
$$[{\bf e}_x,{\bf e}_y]=-\frac{1}{\sqrt{2}}{\bf e}_z;~~~[{\bf e}_x,{\bf e}_z]=-\frac{1}{\sqrt{2}}{\bf e}_y; ~~~[{\bf e}_y,{\bf e}_z]=\frac{1}{\sqrt{2}}{\bf e}_x,$$
and the adjoint operators are
$$\mbox{ad}~{\bf e}_x=-\frac{1}{\sqrt{2}}\left(\begin{array}{ccc}
0&0&0\\
0&0&1\\
0&1&0
\end{array}\right);~~~\mbox{ad}~{\bf e}_y=\frac{1}{\sqrt{2}}\left(\begin{array}{ccc}
0&0&1\\
0&0&0\\
1&0&0
\end{array}\right);~~~\mbox{ad}~{\bf e}_z=\frac{1}{\sqrt{2}}\left(\begin{array}{ccc}
0&-1&0\\
1&0&0\\
0&0&0
\end{array}\right).$$
In this new base the Killing metric becomes 
$${\bf k}(x,y,z)=dx\otimes dx+dy\otimes dy-dz\otimes dz,$$
where $(x,y,z)$ are the coordinates on $\mathbb{R}^3$ corresponding to the base $\{{\bf e}_x,{\bf e}_y,{\bf e}_z\}$.
The gradient vector fields with respect to the Killing metric corresponding to coordinate functions are
$$\nabla x={\bf e}_x;~~~\nabla y={\bf e}_y;~~~\nabla z=-{\bf e}_z.$$
Using the formula \eqref{2} for the Lie-Poisson bracket on $sl(2,\mathbb{R})$ we have
\begin{align*}
\{x,y\}&={\bf k}((x,y,z),[\nabla x,\nabla y])={\bf k}(x{\bf e}_x+y{\bf e}_y+z{\bf e_z},[{\bf e}_x,{\bf e}_y])\\
&={\bf k}\left(x{\bf e}_x+y{\bf e}_y+z{\bf e_z},-\frac{1}{\sqrt{2}}{\bf e}_z\right)=\frac{1}{\sqrt{2}}z.
\end{align*}
Analogously, 
$$\{x,z\}=\frac{1}{\sqrt{2}}y;~~~\{y,z\}=-\frac{1}{\sqrt{2}}x.$$
Consequently, the matrix associated to the Poisson tensor is
$$\boldsymbol{\Pi}=\frac{1}{\sqrt{2}}\left(\begin{array}{ccc}
0&z&y\\
-z&0&-x\\
-y&x&0
\end{array}\right),$$
and a Casimir function is
$C(x,y,z)=x^2+y^2-z^2$.\\
The matrix corresponding to the double bracket contravariant tensor ${\bf D}$ is given by
$${\bf D}=[\boldsymbol{\Pi}^T][{\bf k}][\boldsymbol{\Pi}]=\frac{1}{2}\left(\begin{array}{ccc}
-y^2+z^2&xy&xz\\
xy&-x^2+z^2&yz\\
xz&yz&x^2+y^2
\end{array}\right).$$
Let $\mathbb{H}^2$ be the upper-sheet of the two-sheeted hyperboloid $x^2+y^2-z^2=-1$ and $\mathbb{D}^2=\{(u,v)\in \mathbb{R}^2|~u^2+v^2<1\}$ be the open disc in $\mathbb{R}^2$.
Using the stereographic projection with center $(0,0,-1)$ we obtain the set of local coordinates for $\mathbb{H}^2$, $\Phi:\mathbb{D}^2\rightarrow \mathbb{H}^2\subset \mathbb{R}^3$:
$$\left\{
\begin{array}{l}
u=\frac{x}{1+z}\\
v=\frac{y}{1+z}.
\end{array}
\right.$$
The two 1-forms $du=\frac{1}{1+z}dx-\frac{x}{(1+z)^2}dz$ and $dv=\frac{1}{1+z}dy-\frac{y}{(1+z)^2}dz$ generate a base for the cotangent space $T_{(x,y,z)}^*\mathbb{H}^2$, where $(x,y,z)\in \mathbb{H}^2$. We have the following computations:
\begin{align*}
{\bf D}(x,y,z)(du,du)&=\frac{1}{2}\left(\frac{1}{1+z},0,-\frac{x}{(1+z)^2}\right)\left(\begin{array}{ccc}
1+x^2&xy&xz\\
xy&1+y^2&yz\\
xz&yz&z^2-1
\end{array}\right)\left(\begin{array}{c}
\frac{1}{1+z}\\
0\\
-\frac{x}{(1+z)^2}
\end{array}\right)\\
&=\frac{1}{2(1+z)^2}.
\end{align*}
Analogously, 
$${\bf D}(x,y,z)(du,dv)={\bf D}(x,y,z)(dv,du)=0,~{\bf D}(x,y,z)(dv,dv)=\frac{1}{2(1+z)^2}.$$
Consequently, the matrix associated to the contravariant tensor ${\bf D}$ restricted to the adjoint orbit $\mathbb{H}^2$ is given by
$$\left[{\bf D}(x,y,z)_{|_{T_{(x,y,z)}^*\mathbb{H}^2\times T_{(x,y,z)}^*\mathbb{H}^2}}\right]=\frac{1}{2}\left(\begin{array}{cc}
\frac{1}{(1+z)^2}&0\\
0&\frac{1}{(1+z)^2}
\end{array}\right)$$
and 
$$\left[{\bf D}(x,y,z)_{|_{T_{(x,y,z)}^*\mathbb{H}^2\times T_{(x,y,z)}^*\mathbb{H}^2}}\right]^{-1}=\left[\boldsymbol{\tau}_{db}^{\mathbb{H}^2}(x,y,z)\right]=2\left(\begin{array}{cc}
(1+z)^2&0\\
0&(1+z)^2
\end{array}\right).$$
Having $z=\frac{1+u^2+v^2}{1-u^2-v^2}$ we obtain
$$\boldsymbol{\tau}_{db}^{\mathbb{H}^2}(u,v)
=\frac{8}{[1-(u^2+v^2)]^2}(du^2+dv^2),$$
which is twice the hyperbolic metric for the Poincar\'{e} open disc model. If we drop the coefficient $\frac{1}{\sqrt{2}}$ in the expression of the Poisson tensor $\boldsymbol{\Pi}$, then we obtain the hyperbolic metric.

Note that the induced metric on $\mathbb{H}^2\subset (\mathbb{R}^3,{\bf k})$ is
$${\bf k}_{ind}^{\mathbb{H}^2}(u,v)=\frac{4}{[1-(u^2+v^2)]^2}(du^2+dv^2),$$
which is the hyperbolic metric for the Poincar\'{e} open disc model.

Next we show that the induced metric and double bracket metric differ on the upper-sheet of the two-sheeted hyperboloids $\mathbb{H}^2_c:=\{(x,y,z)\in\R^3|x^2+y^2-z^2=-c^2\}\subset (\mathbb{R}^3,{\bf k})$. Using the following set of local coordinates for $\mathbb{H}^2_c$:
$$\left\{
\begin{array}{l}
x=c \sinh(\nu)\cos(u)\\
y=c \sinh(\nu)\sin(u)\\
z=c \cosh(\nu)
\end{array}
\right. ,$$
we obtain 
$$\boldsymbol{\tau}_{db}^{\mathbb{H}^2_c}(u,\nu)
=2[\cosh(\nu)^2-1]du^2+2d\nu^2,$$ and

$${\bf k}_{ind}^{\mathbb{H}^2_c}(u,\nu)=c^2[\cosh(\nu)^2-1]du^2+c^2d\nu^2.$$

We will repeat the above computations for the regular adjoint orbits given by the one-sheeted hyperboloid
$\mathcal{H}_l:=\{(x,y,z)\in\R|x^2+y^2-z^2=l^2\}\subset (\mathbb{R}^3,{\bf k})$. Using the following set of local coordinates for $\mathcal{H}_l$:
$$\left\{
\begin{array}{l}
x=l \cosh(\nu)\cos(u)\\
y=l \cosh(\nu)\sin(u)\\
z=l \sinh(\nu)
\end{array}
\right. ,$$
we obtain 
$$\boldsymbol{\tau}_{db}^{\mathcal{H}_l}(u,\nu)
=-2\cosh^2(\nu)du^2+2d\nu^2,$$ and
$${\bf k}_{ind}^{\mathcal{H}_l}(u,\nu)=l^2\cosh^2(\nu)du^2-l^2d\nu^2.$$
The semi-simple Lie algebra $sl(2,\mathbb{R})$ has another two regular orbits given by the connected components of the cone ${\cal C}=\{(x,y,z)\in \mathbb{R}^3|x^2+y^2-z^2=0\}$ without its vertex $(0,0,0)$. On this two symplectic leaves the induced metric by the pseudo-Riemannian metric ${\bf k}$ is degenerate and the construction given in Definition \ref{double-bracket-metric} does not apply.

\medskip

\noindent {\bf Acknowledgements.} This work was supported by a grant of the Romanian National Authority for Scientific
Research, CNCS - UEFISCDI, project number PN-II-RU-TE-2011-3-0006.

\end{document}